\documentclass[12pt]{article}
  \usepackage[centertags]{amsmath}
  \usepackage{amsfonts}
  \usepackage{amssymb}
  \usepackage{amsthm}
  \usepackage{newlfont}
\usepackage{amssymb}%
\usepackage{graphicx}

  \newlength{\defbaselineskip}
  \newcommand{\setlinespacing}[1]%
                               {\setlenght{\baselineskip}{#1 \defbaselineskip}}

  \def\R{\mbox{I\hspace{-.15em}R}}

  \theoremstyle{plain}

\newtheorem{theorem}{Theorem}[section]

\newtheorem{proposition}[theorem]{Proposition}
\newtheorem{example}[theorem]{Example}

\theoremstyle{remark}
\newtheorem{remark}[theorem]{Remark}
\newtheorem{definition}{Definition}[section]
\numberwithin{equation}{section}

\begin{document}
\begin{center}
{\bf The Apple Doesn't Fall Far From the (Metric) Tree: The Equivalence of Definitions}
\end{center}
\vspace{.15 cm}
\begin{center}
\small{Asuman G. AKSOY and Sixian JIN}
\end{center}

\date{March 20, 2013}
\mbox{~~~}\\
\mbox{~~~}\\
\small\mbox{~~~~}{\bf Abstract.} {\footnotesize In this paper we prove the equivalence of definitions for metric trees and for $\delta$-hyperbolic spaces. We point out how these equivalences can be used to understand the geometric and metric properties of $\delta$-hyperbolic spaces and its relation to CAT($\kappa$) spaces.}
\\

\section{ Introduction}

A metric space is a metric tree  if and only if it is $0$-hyperbolic and geodesic. In other words, a geodesic metric space  is said to be a metric tree (or an $\mathbb{R}$-tree, or a T-tree) if it is $0$-hyperbolic in the sense of Gromov that all of its geodesic triangles are isometric to tripods. It is well known that every $0$-hyperbolic metric space embeds isometrically into a metric tree (see \cite{Dr}, \cite{godard}) and the construction of metric trees is related to the asymptotic geometry of hyperbolic spaces (see \cite{Brid}, \cite {DP}). Metric trees are not only described by different names but are also given by different definitions. In the following, we state two widely used definitions of a metric tree:

\begin{definition}\label{Def1}
  An $\mathbb{R}$-tree is a metric space $M$ such that for every
$x$ and $y$ in $M$ there is a unique arc between $x$ and $y$ and this arc is
isometric to an interval in $\mathbb{R}$ (i.e., is a geodesic segment).\\
\end{definition}
 Recall that for $x,y \in M$ a \emph{geodesic segment} from $x$ to $y$  denoted by $[x,y]$ and is the image of an isometric embedding $\alpha : [a,b]\rightarrow M$ such that $\alpha(a)=x$ and $\alpha (b)=y$.  A geodesic metric space is a metric space in which every pair of points is joined by a (not necessarily unique) geodesic segment.

\begin{definition}\label{Def2}

An $\mathbb{R}$-tree is a metric space $M$ such that\medskip

(i) there is a unique geodesic segment denoted by $[x,y]$ joining each pair of
points $x$ and $y$ in M;\medskip

(ii) if $\left[  y,x\right]  \cap\left[  x,z\right]  =\left\{  x\right\}
\Rightarrow\left[  y,x\right]  \cup\left[  x,z\right]  =\left[  y,z\right]  .$\\
\end{definition}

Condition (ii) above simply states that if two segments intersect in a single point then their union is a segment too. Note that $\mathbb{R}^n$ with the Euclidean metric satisfies the first condition. It fails, however, to satisfy the second condition.\\
The study of  metric trees is motivated by many subdisciplines of mathematics \cite{Ev}, \cite{Tits}, biology, medicine
and computer science. The relationship between metric trees and biology and medicine
stems from the construction of phylogenetic trees \cite{Semp}; and
concepts of ``string matching" in computer science are closely
related with the structure of metric trees \cite{Bart}.
Unlike metric trees, in an ordinary tree  all the edges are assumed to have
the same length and therefore the metric is not often stressed. However, a
metric tree is a generalization of an ordinary tree that allows
for different edge lengths.  For example, a connected graph without
loops is a metric tree.  Metric trees also arise naturally in the study of group isometries of hyperbolic
spaces. For metric properties of trees we refer to \cite{buneman}. Lastly, \cite{mmot} and \cite{mo} explore the topological characterization of metric trees.
 For an overview of geometry, topology, and group theory applications of metric trees, consult \cite{Best}. For a complete discussion of these spaces
and their relation to $CAT (\kappa)$ spaces, see the well known monograph by Bridson and Haefliger \cite{Brid}. Recall that a complete geodesic metric space is said to be a \emph{CAT}($\kappa$) space (or a Hadamard space) if it is geodesically connected and if every geodesic triangle in $X$  is at least as ``thin" as its comparison triangle in, respectively, the classical spherical space $\mathbb{S}^2_{\kappa}$ of curvature $\kappa$ if  $\kappa > 0$, the Euclidean plane if $\kappa= 0$, and the classical hyperbolic space of curvature $\kappa$ if $\kappa < 0$.\\
Given a metric $d(x,y)$, we denote it by $xy$. We also say that a point $z$ is \textit{between} $x$ and $y$ if $xy = xz + zy$. We will often denote this by $xzy$. It is not difficult to prove that in any metric space, the elements of a metric segment from $x$ to $y$ are necessarily between $x$ and $y$, and in a metric tree, the elements between $x$ and $y$ are the elements in the unique metric segment from $x$ to $y$. Hence, if $M$ is a metric tree and $x,y \in M$, then $$[x,y] = \{ z \in M : xy = xz + zy \}.$$  The following is an example of a metric tree. For more examples see \cite{AkTi}.

\begin{example}(The Radial Metric)\label{E:radial}
    Define $d: \mathbb{R}^2 \times \mathbb{R}^2 \to \mathbb{R}^+$ by:
    \[
        d(x,y) =
            \begin{cases}
            \|x-y\| & \text{if $x = \lambda \, y$ for some $\lambda \in \R$,}\\
            \|x\| + \| y \| & \text{otherwise.}
            \end{cases}
    \]
    We can observe that the $d$ is in fact a metric and that $(\mathbb{R}^2,d)$ is a metric tree.
\end{example}
 It is well known that any complete, simply connected Riemannian manifold having non-positive curvature is a $CAT(0)$-space. Other examples include the complex Hilbert ball with the hyperbolic metric (see \cite{Goebel}), Euclidean buildings (see \cite{Brown}) and classical hyperbolic spaces.  If a space is $CAT(\kappa)$ for some $\kappa < 0$ then it is automatically a $CAT(0)$-space. In particular, metric trees are a sub-class of $CAT(0)$-spaces, and we note the following:
 \begin{proposition}
 If a metric space is $CAT(\kappa)$ space for all $\kappa$, then it is a metric tree.
 \end{proposition}

 For the proof of the above proposition, see p. $159$ of \cite{Brid}. Note that if a Banach space is a $CAT(\kappa)$  space for some $\kappa$ then it is necessarily a Hilbert space and $CAT(0)$. The property that distinguishes the metric trees from the $CAT(0)$ spaces is the fact that metric trees are  hyperconvex metric spaces. Properties of hyperconvex spaces and their relation to metric trees can be found in \cite{Akso}, \cite{ap}, \cite{isbell} and \cite{KirkH}.
We refer to \cite{Blum} for the properties of metric segments and to \cite{AkBo} and \cite{AkKh} for
 the basic properties of complete metric trees. In the following we list some of the properties of metric trees which will be used in the proof of Theorem 2.1.
\begin{enumerate}
 \item  (Uniform Convexity \cite{AkTi}). A metric tree $M$ is uniformly convex.
 \item (Projections are nonexpansive \cite{AkKh}). Metric projections on closed convex subsets of a metric tree are nonexpansive.
\end{enumerate}
Property $1$ above generalizes the classical Banach space notion of uniform convexity by defining the modulus of convexity for geodesic metric spaces. Let $C$ be a closed convex  subset (by convex we mean for all $x,y \in C$, we have $[x,y] \subset C$) of a metric tree $M$. If for every point $x\in M$ there exists a nearest point in $C$ to $x$, and if this point is unique, we denote this point by $P_{C}(x)$, and call the mapping $P_C$ the metric projection from $M$ into $C$. In Hilbert spaces, the metric projections on closed convex subsets are nonexpansive. In uniformly convex spaces, the metric projections are uniformly Lipschitzian.  In fact, they are nonexpansive if and only if the space is Hilbert. Property $2$ is remarkable in this context and this result is not known in hyperconvex spaces. However, the fact that the nearest point projection onto convex subsets of metric trees is nonexpsansive also follows from the fact that this is true in the more general setting of $CAT(0)$ spaces (see p. 177 of \cite {Brid}). We will use above properties in the proof of Theorem \ref{thm1}\\
A metric space $\left( X,d\right) $ is said to have the \emph{four-point property}
 if for each $x,y,z,p\in
X$\,$$ d\left( x,y\right) +d\left( z,p\right) \leq \max \left\{ d\left(
x,z\right) +d\left( y,p\right) ,d\left( x,p\right) +d\left( y,z\right)
\right\} $$ holds. The four-point property characterizes metric trees (see \cite{Akso}) thus, its natural extension characterizes $\delta$-hyperbolic spaces as seen in Definition \ref{Def4} below.\\
 In the following, we give three widely used definitions of $\delta$-hyperbolic spaces and references to how these definitions are utilized in order to describe geometric properties.\\
\begin{definition}\label{Def3}

 A metric space $(X,d)$ is  $\delta$-hyperbolic if for all $p,x,y,z\in X$,
\begin{equation}\label{def1}
\left( x,z\right) _{p}\geq \min \left\{ \left( x,y\right)
_{p},\left( y,z\right) _{p}\right\} -\delta
\end{equation}

where
$\left( x,z\right) _{p}\ =\frac{1}{2}\left( d\left( x,p\right)
+d\left( z,p\right) -d\left( x,z\right) \right)$  is the Gromov product.
\end{definition}
\bigskip
\begin{definition}\label{Def4}
 A metric space $\left( X,d\right) $ is
called $\delta -hyperbolic\ $for$\ \delta \geq 0$ if for each $x,y,z,p\in
X,\ d\left( x,y\right) +d\left( z,p\right) \leq \max \left\{ d\left(
x,z\right) +d\left( y,p\right) ,d\left( x,p\right) +d\left( y,z\right)
\right\} +2\delta$.
\end{definition}
\bigskip

\begin{definition}\label{Def5}
 A geodesic metric space $\left( X,d\right) \
$is$\ \delta -hyperbolic\ $if$\ $every geodesic triangle is $\delta -thin,
i.e., $given a geodesic triangle $\triangle xyz\subset X,\ \forall a\in %
\left[ x, y\right] ,\exists b\in\left[ x, z\right] \cup \left[ z, y\right] \ $%
such that$\ d\left( a,b\right) \leq \delta .\ \left[ x, y\right] $ is the
geodesic segment joint $x,y.$
\end{definition}
\bigskip

 Definition \ref{Def3} is the original definition for $\delta$-hyperbolic
spaces from Gromov in \cite{Gromov}, which depends on the notion of Gromov product. The Gromov product measures the failure of the
triangle inequality to be an equality. This definition appears in almost
every paper where $\delta$-hyperbolic spaces are discussed. Although one can provide a long list from our references we refer the reader to \cite{Vaisala}
,\cite{BonkSch},  \cite{Ibr1}, \cite{Ibr2}, \cite{godard} and \cite{Brid}.
 The Gromov product enables one to define ``convergence" at infinity and by this convergence the boundary of $X$, $\partial{X}$, can be defined. The metric on $\partial{X}$  is the so called
 ``visual metric" (see \cite{BonkSch} and \cite{Brid}). The advantage of Definition \ref{Def3} is that it facilitates the relationship between maps of $\delta$-hyperbolic spaces and maps of their boundary \cite {BonkSch},
\cite{Julian}. \\

Definition \ref{Def4} is a generalization of famous four-point property for which   $%
\delta =0$. The four-point property  plays an important role in metric
trees, for example, in \cite{Akso}, it is
shown that a metric space is a metric tree if and only if it is complete,
connected and satisfies the four-point property. However, it is also well known that a complete geodesic metric space $X$ is a CAT(0) if and only if it satisfies the four-point condition (see \cite{Brid}). Furthermore, in \cite{godard} Godard proves that for a given metric space $M$, each Lipschitz-free space $F(M)$ is isometric to a subspace of $L_1$. This is equivalent to $M$ satisfying four-point condition, and
 the fact that $M$ isomerically embedds into a metric tree.
The advantage of Definition \ref{Def4} is that we can write out the inequality
directly by distance of the metric space instead of by the Gromov product. In some
cases if we construct a metric with the distance function having a
particular form, it is easier to deal with distance inequality than Gromov
product inequality. For example, in \cite {Ibr1}, \cite{Ibr2} Ibragimov provides a method to
construct a Gromov hyperbolic space by ``hyperbolic filling" under a
proper compact ultrametric space and such a ``filling"  of a space contains points
which are metric balls in original ultrametric space and is equipped with a
distance function $h\left( A,B\right) =2\log \displaystyle \frac{diam\left( A\cup B\right)
}{\sqrt{diam\left( A\right) diam\left( B\right) }}.$ 

Note that the Definition \ref{Def5} of $\delta$-hyperbolic spaces requires that the underlying space is geodesic since it depends on geodesic triangles. Yet in \cite{BonkSch} , Bonk and
Schramm show that any $\delta$-hyperbolic space can be isometrically
embedded into a geodesic $\delta$-hyperbolic space. Thus one has the freedom of using Definition \ref{Def5}.\\

Furthermore, recall that we call $X$ \emph{hyperbolic} if it is $\delta$-hyperbolic for some $\delta \geq 0$. Sometimes $\delta$ is referred as a \emph{hyperbolicity constant} for $X$. Besides any tree being $0$-hyperbolic, any space of finite diameter, $\delta$, is $\delta$-hyperbolic and the hyperbolic plane $\mathbb{H}^2$ is $(\displaystyle \frac{1}{2} \log3)$-hyperbolic. In fact any simply connected Riemanian manifold with curvature bounded above by some negative constant  $-\kappa^2 < 0$ is $(\displaystyle \frac{1}{2\kappa} \log3)$-hyperbolic (see \cite{Brid}).

\bigskip



\section{Main Results }

\begin{theorem}\label{thm1}
Definition \ref{Def1}  and Definition \ref{Def2}  of metric trees are equivalent.
\end{theorem}

\begin{proof}
Suppose $M$ is a $\mathbb{R}$-tree in the sense of Definition \ref{Def1}, and let
$x,y\in M.$ Then by Definition \ref {Def1}, there is a unique arc joining $x$ and $y$
which is isometric to an interval in $\mathbb{R}.$ Hence it is a geodesic
(i.e., metric) segment. So we may denote it by $\left[  x,y\right]  .$ Thus we
have defined a unique metric segment $[x,y]$ $\left(  =\left[  y,x\right]
\right)  $ for each $x,y\in M,$ so (i) holds.

To see that (ii) holds, suppose $\left[  y,x\right]  \cap\left[  x,z\right]
=\left\{  x\right\}  .$ Then, $\left[  y,x\right]  \cup\left[  x,z\right]  $ is
an arc joining $y$ and $z;$ and by Definition \ref{Def1} it must be isometric to a
real line interval. Therefore it must be precisely the unique metric segment
$\left[  y,z\right]$. Now suppose $M$ is a $\mathbb{R}$-tree in the sense of Definition \ref{Def2}, and
let $x,y\in M.$ Then $\left[  x,y\right]  $ is an arc joining $x$ and $y,$ and
it is isometric with a real line interval. We must show that this is the only
arc joining $x$ and $y.$

 Suppose $A$ is an arc joining $x$ and $y,$ with
$A\neq\left[  x,y\right]  .$ By passing to a subarc, if necessary, we may
without loss of generality, assume $A\cap\left[  x,y\right]  =\left\{
x,y\right\}  .$ Let $P$ be a nonexpansive projection of $M$ onto $\left[
x,y\right]  .$ Since $P$ is continuous with $P\left(  x\right)  =x$ and
$P\left(  y\right)  =y,$ clearly there must exist $z_{1},z_{2}\in
A\backslash\left\{  x,y\right\}  $ such that $P\left(  z_{1}\right)  \neq
P\left(  z_{2}\right)  .$ Let $A_{1}$ denote the subarc of $A$ joining $z_{1}$
and $z_{2}.$ Fix $z\in A_{1}.$ If $u\in A_{1}$ satisfies $d\left(  u,z\right)
<d\left(  z,P\left(  z\right)  \right)  ,$ then it must be the case that
$P\left(  u\right)  =P\left(  z\right)  .$ Here we use the fact that $\left[
x,P\left(  z\right)  \right]  \cap\left[  P\left(  z\right)  ,z\right]
=\left\{  P\left(  z\right)  \right\}  ;$ hence by (ii) $\left[  x,z\right]
=\left[  x,P\left(  z\right)  \right]  \cup\left[  P\left(  z\right)
,z\right]  .$  Therefore, there is an open
neighborhood $N_{z}$ of $z$ such that $u\in N_{z}\cap A_{1}\Rightarrow
P\left(  u\right)  =P\left(  z\right)  .$ The family $\left\{  N_{z}\right\}
_{z\in A}$ covers $A_{1},$ so, by compactness of $A_{1}$ there exist $\left\{
z_{1},\cdot\cdot\cdot,z_{n}\right\}  $ in $A_{1}$ such that $A_{1}\subset%
{\displaystyle\bigcup_{i=1}^{n}}
N_{z_{i}}.$ However, this implies $P\left(  z_{1}\right)  =P\left(
z_{2}\right)  $ which is a contradiction. Therefore, $A=\left[  x,y\right]  ,$ and
since $\left[  x,y\right]  $ is isometric to an interval in $\mathbb{R}$, the
conditions of Definition \ref{Def1} are fulfilled.
\end{proof}
\begin{remark}
In the above proof we used the fact that the closest point projection onto a closed metrically convex subset is nonexpansive.  Definition \ref{Def2} is used in fixed point theory, mainly to investigate and see whether much of the known results for nonexpansive mappings remain valid in complete $CAT(0)$ spaces with asymptotic centre type of arguments used to overcome the lack of weak topology. For example it is shown that if $C$ is a nonempty connected bounded open subset of a complete $CAT(0)$ space $(M,d)$ and $T: \overline{C}\rightarrow M$ is nonexpansive, then either
\begin{enumerate}
\item $T$ has a fixed point in $\overline{C}$, or
\item $ 0 < \inf\{d(x,T(x)): \,\,\, x\in \partial {C} \}. $
\end{enumerate}
Application of these to metrized graphs has led to "topological" proofs of graph theoretic results; for example refinement of the fixed edge theorem (see \cite{KirkH}, \cite{EK}, \cite{Kirk}). Definition \ref{Def1} used to construct T-theory and its relation to tight spans (see \cite {DMT}) and best approximation in $\mathbb{R}$-trees (see \cite{KP}).
\end{remark}
\begin{theorem}
{Definition \ref{Def3}}, {Definition \ref{Def4}} and {Definition \ref{Def5}} of $\delta$-hyperbolic spaces are equivalent.
\end{theorem}

\begin{proof}
We suppose $X$ be a geodesic Gromov $\delta -hyperbolic$ space below. We first show Definition\ref{Def3} implies Definition\ref{Def4}.
 By Definition\ref{Def3}, we have
\begin{equation*}
d\left( x,p\right) +d\left( y,p\right) -d\left( x,y\right) \geq \min \left\{
d\left( x,p\right) +d\left( z,p\right) -d\left( x,z\right) ,d\left(
y,p\right) +d\left( z,p\right) -d\left( y,z\right) \right\} -2\delta .
\end{equation*}

\noindent Without loss of generality we can suppose%
\begin{equation*}
d\left( x,p\right) +d\left( z,p\right) -d\left( x,z\right) \geq d\left(
y,p\right) +d\left( z,p\right) -d\left( y,z\right)
\end{equation*}%
$i.e.$%
\begin{equation*}
d\left( x,p\right) +d\left( y,z\right) \geq d\left( y,p\right) +d\left(
x,z\right) .
\end{equation*}

\noindent So we have $d\left( x,p\right) +d\left( y,p\right) -d\left( x,y\right) \geq
d\left( y,p\right) +d\left( z,p\right) -d\left( y,z\right) -2\delta \ $  or equivalently \ $d\left( x,p\right) +d\left( y,z\right) +2\delta \geq d\left(
z,p\right) +d\left( x,y\right) .$

\noindent The same conclusion follows if we take
\begin{equation*}
d\left( x,p\right) +d\left( z,p\right) -d\left( x,z\right) \leq d\left(
y,p\right) +d\left( z,p\right) -d\left( y,z\right)
\end{equation*}%
and we have $d\left( z,p\right) +d\left( x,y\right) \leq \max \left\{
d\left( x,y\right) +d_{xz},d_{xp}+d_{yz}\right\} +2\delta .$

\bigskip

\noindent To show Definition\ref{Def4} implies Definition \ref{Def3},without loss of generality, we suppose

\begin{equation*}
d\left( x,z\right) +d\left( y,p\right) \leq d\left( x,p\right) +d\left(
y,z\right) .
\end{equation*}

\noindent So, we have
\begin{equation*}
d\left( x,y\right) +d\left( z,p\right) \leq d\left( x,p\right) +d\left(
y,z\right) +2\delta .
\end{equation*}

\noindent Then
\begin{eqnarray*}
d\left( y,p\right) +d\left( z,p\right) -d\left( y,z\right) &\leq &d\left(
x,p\right) +d\left( z,p\right) -d\left( x,z\right) \\
d\left( y,p\right) +d\left( z,p\right) -d\left( y,z\right) -2\delta &\leq
&d\left( x,p\right) +d\left( y,p\right) -d\left( x,y\right)
\end{eqnarray*}
and we get
\begin{equation*}
d\left( x,p\right) +d\left( y,p\right) -d\left( x,y\right) \geq \min \left\{
d\left( y,p\right) +d\left( z,p\right) -d\left( y,z\right) ,d\left(
x,p\right) +d\left( z,p\right) -d\left( x,z\right) \right\} -2\delta .
\end{equation*}%

\bigskip

\noindent To prove equivalence of Definition \ref{Def3} and Definition \ref{Def5}, we need the following property:

\noindent For any geodesic triangle $\triangle xyz$
in metric space $\left( M,d\right) \ $we can find three points on each side
denoted by $a_{x}\ $on $\left[ y, z\right] ,a_{y}$ on $\left[ x, z\right] $ and $%
a_{z}$ on $\left[ x, y\right] \ $such that

\begin{eqnarray*}
d\left( a_{x},y\right) &=&d\left( a_{z},y\right) =\left( x,z\right) _{y} \\
d\left( a_{z},x\right) &=&d\left( a_{y},x\right) =\left( y,z\right) _{x} \\
d\left( a_{y},z\right) &=&d\left( a_{x},z\right) =\left( x,y\right) _{z}.
\end{eqnarray*}%

\noindent To show Definition\ref{Def5} implies Definition\ref{Def3}, for any $x,y,p\in X,$ we will show that for any $z\in X $
following holds:

\begin{equation*}
\left( x,y\right) _{p}\geq \min \left\{ \left( x,z\right) _{p},\left(
z,y\right) _{p}\right\} -3\delta .
\end{equation*}

\begin{center}
\includegraphics [scale=.26]  {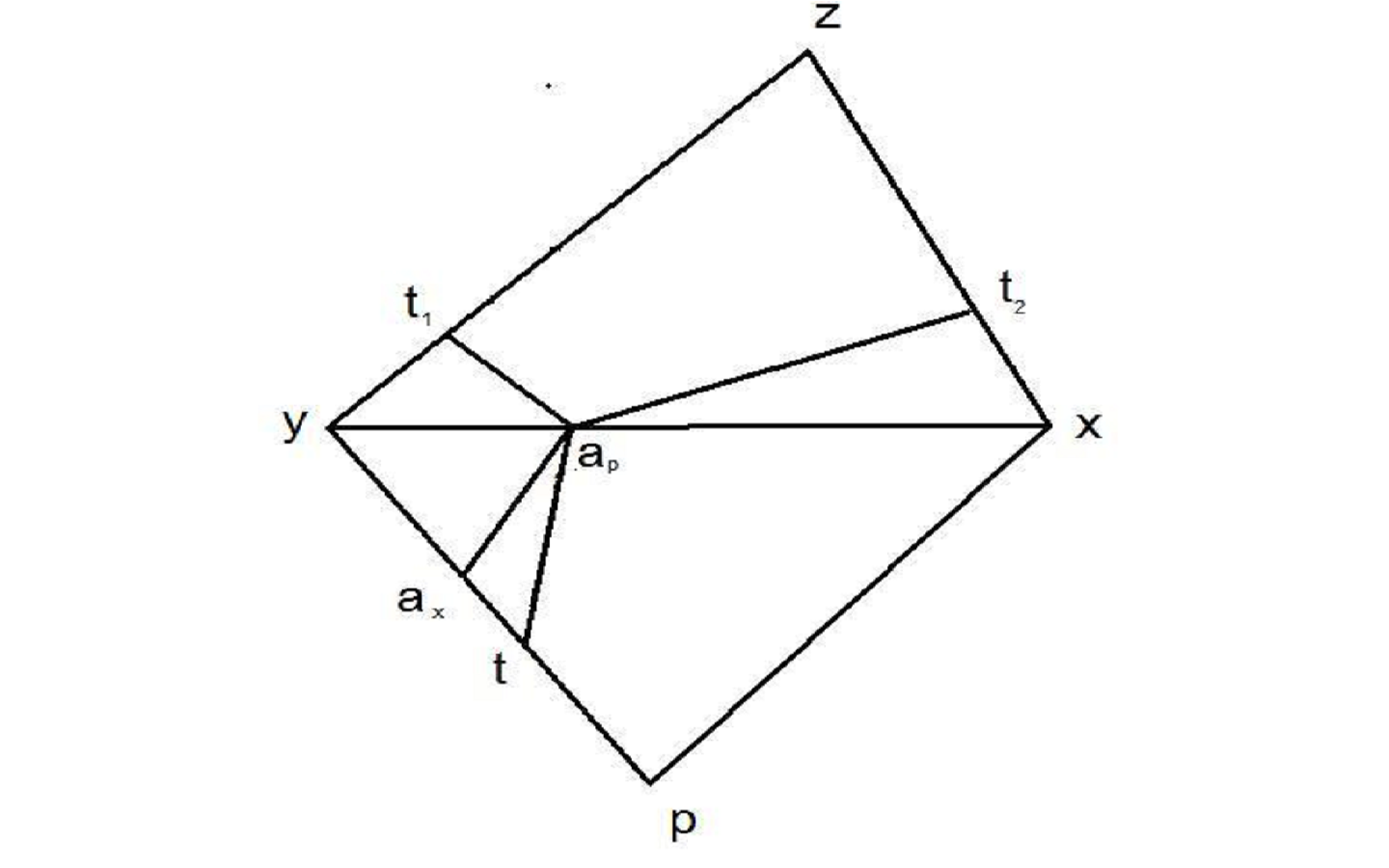}

\end{center}

\begin{center}
Figure $1$
\end{center}


\noindent By the above stated property, in triangle $\triangle xyp\ $ we choose
three points $a_{p},a_{x}$ and $a_{y}$ in $\left[ x, y\right] ,\left[ p, y\right]
$ and$\ \left[ p, x\right] $ as shown in Figure 2 such that

\begin{eqnarray*}
d\left( y,a_{p}\right) &=&\left( p,x\right) _{y}=d\left( y,a_{x}\right) \\
d\left( x,a_{p}\right) &=&\left( p,y\right) _{x}=d\left( x,a_{y}\right) \\
d\left( p,a_{x}\right) &=&\left( x,y\right) _{p}=d\left( p,a_{y}\right)
\end{eqnarray*}

\noindent Without loss of generality we assume $t\in \left[ p, y%
\right] \ $such that $d\left( a_{p},t\right) \leq \delta $ and \\$d\left(
t,y\right) >d\left( a_{x},y\right) ,\ $then in triangle $\triangle yta_{p},$
\begin{equation*}
d\left( t,y\right) <d\left( t,a_{p}\right) +d\left( a_{p},y\right) =\delta
+d\left( a_{x},y\right)
\end{equation*}%
so$\ d\left( t,a_{x}\right) <\delta $ and the same conclusion follows if we suppose $d\left( t,y\right)
<d\left( a_{x},y\right) $.

\noindent Then for $\triangle a_{p}a_{x}t,$
\begin{eqnarray*}
d\left( a_{p},a_{x}\right)
<d\left( t,a_{x}\right) +d\left( a_{p},t\right) <2\delta .
\end{eqnarray*}

\noindent For any $z\in X,$ consider $\triangle xyz$ and choose $t_{1}\in \left[
y, z\right] $ and $t_{2}\in \left[ x, z\right] \ $such that $d\left(
a_{p},t_{1}\right) $ and $d\left( a_{p},t_{2}\right) $ are the shortest
distances from $a$ to $\left[ y, z\right] $ and $\left[ x, z\right]$, therefore
\begin{equation*}
\min \left\{ d\left( a,t_{1}\right) ,d\left(
a,t_{2}\right) \right\} \leq \delta .
\end{equation*}

\noindent Then looking at triangles $\triangle pa_{p}t_{1}$ and $\triangle pa_{p}t_{2}$ we have
\begin{eqnarray*}
\min \left\{ d\left( p,t_{1}\right) ,d\left( p,t_{2}\right) \right\} \leq
\min \left\{ d\left( a_{p},t_{1}\right) ,d\left( a_{p},t_{2}\right) \right\}
+d\left( p,a_{p}\right) \\
\leq \delta +d\left( p,a_{p}\right) \leq \delta
+d\left( p,a_{x}\right) +d\left( a_{p},a_{x}\right) \leq 3\delta +\left(
x,y\right) _{p}.
\end{eqnarray*}

\noindent Since
\begin{eqnarray*}
\left( y,z\right) _{p}=\frac{1}{2}\left( d\left( y,p\right)
+d\left( z,p\right) -d\left( y,z\right) \right)=\frac{1}{2}\left( d\left(y,p\right) -d\left( y,t_{1}\right) +d\left( z,p\right) -d\left(
z,t_{1}\right) \right)
\end{eqnarray*}
by triangle inequality we have $\left( y,z\right) _{p}\leq d\left( p,t_{1}\right)$ and similarly for $\left(
x,z\right) _{p}\leq d\left( p,t_{2}\right).$

\noindent Then $\min \left( \left( y,z\right) _{p},\left( x,z\right)
_{p}\right) \leq \min \left\{ d\left( p,t_{1}\right) ,d\left( p,t_{2}\right)
\right\} \leq 3\delta +\left( x,y\right) _{p}.$

\bigskip

\begin{center}
\includegraphics [scale=.2]  {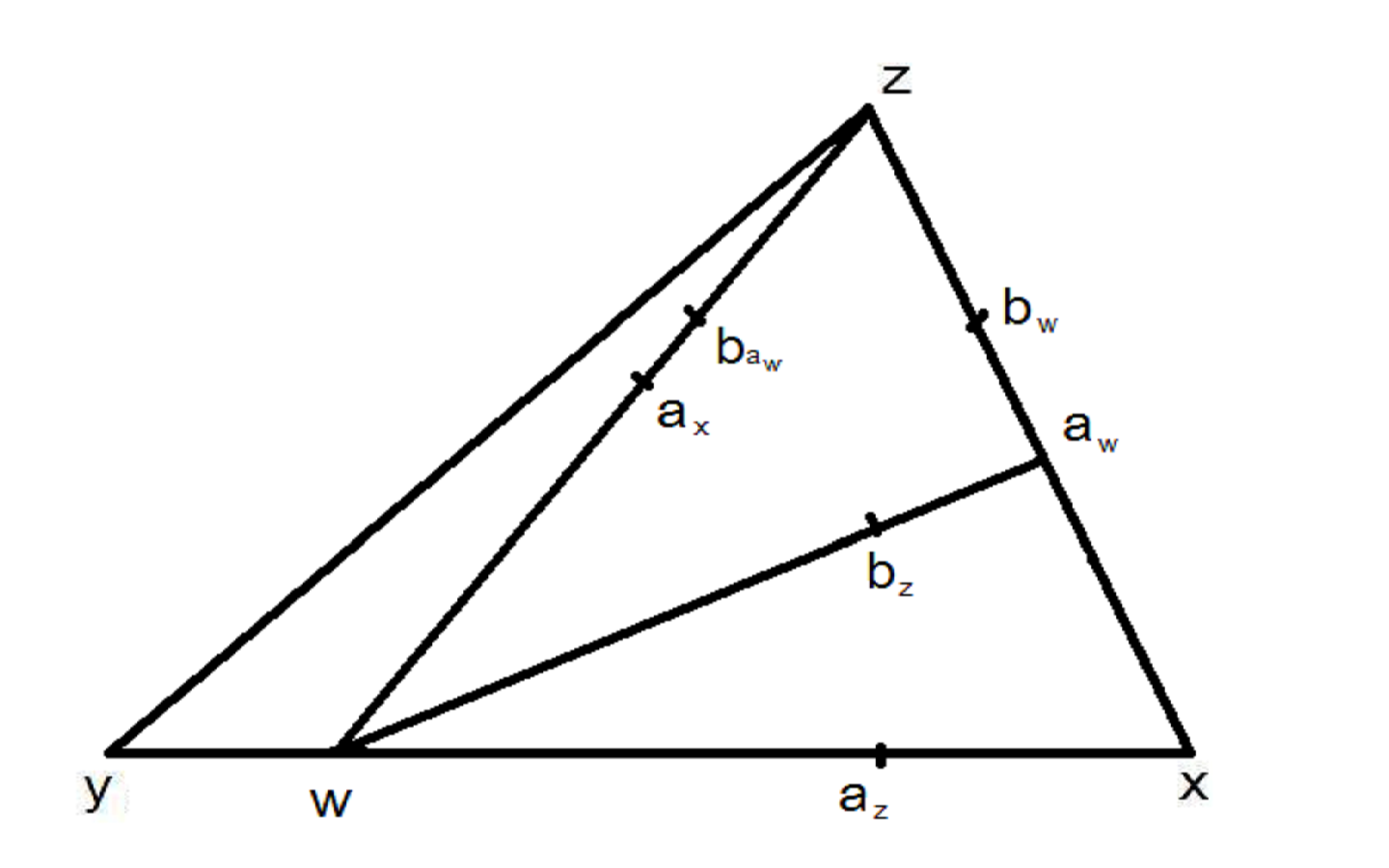}

\end{center}

\begin{center}
Figure $2$
\end{center}

To show Definition \ref{Def3} implies Definition \ref{Def5}, let $\triangle xyz$ be a
geodesic triangle and $w\in \left[ x, y\right] ,$ see Figure 2. Let $d\left( x,%
\left[ y, z\right] \right) $ denote the shortest distance from $x$ to the side $\left[
yz\right] ,$ without loss of generality we assume $\left( x,z\right) _{w}\geq \left( y,z\right) _{w}.$ Then by Definition \ref{Def3}
\begin{equation*}
\left( x,y\right) _{w}\geq \min \left\{ \left( x,z\right) _{w},\left(
y,z\right) _{w}\right\} -\delta =\left( x,z\right) _{w}-\delta \
\end{equation*}%
which implies $\delta \geq \left( x,z\right) _{w}.$

\noindent Next we consider the triangle $\triangle xzw$ and find three
points $a_{x},a_{z}$ and $a_{w}$ on each side with the previous property. Then
\begin{eqnarray*}
d\left( w,a_{x}\right) =\left( x,z\right) _{w}  \\
d\left(w,a_{w}\right) \geq d\left( w,\left[ x ,z\right] \right).
\end{eqnarray*}

\noindent Similarly, in $\triangle xa_{w}z\ $ one can find three points $%
b_{z},b_{w}$ and $b_{a_{w}}$ on $\left[ w, a_{w}\right] ,\left[ a_{w}, z\right] $
and $\left[ z, w\right] ,$ which satisfy the previous property and we have
$d\left( w,b_{a_{w}}\right) =\left(
z,a_{w}\right) _{w}.$

\noindent We assume  $\left( z,a_{w}\right) _{w}<\left(
x,a_{w}\right) _{w}\, $ then $d\left( w,a_{x}\right) $ $<d\left(
w,b_{a_{w}}\right).$

\noindent So,

\noindent $\delta \geq \min \left\{ \left( z,a_{w}\right)
_{w},\left( x,a_{w}\right) _{w}\right\} -\left( x,z\right) _{w}=\left(
z,a_{w}\right) _{w}-\left( x,z\right) _{w}=d\left( b_{a_{w}},w\right)
-d\left( a_{x},w\right)$

\noindent =$d\left( a_{x},b_{a_{w}}\right) =\left( a_{w},z\right) _{w}-\left(
x,z\right) _{w}$

\noindent =$\frac{1}{2}\left( d\left( a_{w},w\right) +d\left( x,z\right)
-d\left( a_{w},z\right) -d\left( x,w\right) \right)$

\noindent =$\frac{1}{2}\left( d\left( a_{w},x\right) +d\left( w,b_{z}\right)
+d\left( a_{w},b_{z}\right) -d\left( w,a_{z}\right) -d\left( a_{w},x\right)
\right)$

\noindent =$\frac{1}{2}\left( d\left( a_{x},b_{a_{w}}\right) +d\left(
a_{w},b_{w}\right) \right)$

\noindent which implies $d\left( a_{x},b_{a_{w}}\right) =d\left(
a_{w},b_{w}\right) .$

\noindent Thus, $d\left( w,a_{w}\right) =d\left( w,b_{z}\right) +d\left(
b_{z},a_{w}\right) =d\left( w,b_{a_{w}}\right) +d\left( a_{w},b_{w}\right)
=d\left( w,a_{x}\right) +2d\left( a_{w},b_{a_{w}}\right) \leq \left(
x,z\right) _{w}+2\delta .$

\noindent Then $d\left( w,\left[ x, z\right] \right) \leq d\left( w,a_{w}\right)
\leq \left( x,z\right) _{w}+2\delta \leq 3\delta .\ \ \ \ \ \ \ \ \ \ \ \ \
\ \ \ \ \ \ \ \ \ \ \ \ \ \ \ \ \ \ \ \ \ \ \ \ \ \ $

\end{proof}


\begin{remark}
In \cite{BonkFoe} Bonk and Foertsch use the inequality $(1.1)$ repeatedly to define a new
space, called  $AC_{u}\left( \kappa \right)$-space by introducing the notion of upper curvature bounds for Gromov hyperbolic spaces. This space is equivalent to a $
\delta$-hyperbolic space and furthermore it establishes a precise relationship between CAT($\kappa)$ spaces and $\delta$-hyperbolic spaces. It is well known that any CAT($\kappa)$ space with negative $\kappa $ is a $\delta$-hyperbolic space for some $\delta$. In \cite{BonkFoe} it is shown that a CAT($\kappa)$ space with negative $\kappa $  is just an $AC_{u}\left( \kappa \right)$-space. Moreover, following the arguments in \cite{BonkFoe}, Fournier, Ismail and
Vigneron in  \cite{Fournier} compute an approximate value for $\delta$ .

\end{remark}



\bibliographystyle{amsplain}



\noindent
\mbox{~~~~~~~}Asuman G\"{u}ven AKSOY\\
\mbox{~~~~~~~}Claremont McKenna College\\
\mbox{~~~~~~~}Department of Mathematics\\
\mbox{~~~~~~~}Claremont, CA  91711, USA \\
\mbox{~~~~~~~}E-mail: aaksoy@cmc.edu \\ \\
\noindent
\mbox{~~~~~~~}Sixian JIN\\
\mbox{~~~~~~~}Claremont Graduate University\\
\mbox{~~~~~~~}Department of Mathematics\\
\mbox{~~~~~~~}Claremont, CA, 91711, USA\\
\mbox{~~~~~~~}E-mail: Sjin@cgu.edu\\\\

\end{document}